\documentclass[10pt,reqno,oneside]{amsproc}
\title[{Observability from a measurable set}]{Observability from a measurable set for functions in a Gevrey class}

\author[I.~Kukavica]{Igor Kukavica}
\address{Department of Mathematics\\
University of Southern California\\
Los Angeles, CA 90089}
\email{kukavica@usc.edu}

\author[L.~Li]{Linfeng Li}
\address{Department of Mathematics\\
University of California Los Angeles\\
Los Angeles, CA 90095}
\email{lli265@math.ucla.edu}

\usepackage{enumitem}
\usepackage{bbm}
\usepackage{fancyhdr}
\usepackage{comment}
\usepackage{tikz}
\chardef\forshowkeys=0
\chardef\refcheck=0
\chardef\showllabel=0
\chardef\sketches=1

\ifnum\forshowkeys=1

  \usepackage[notref,notcite,color]{showkeys}
\fi

\usepackage[margin=1in]{geometry}
\usepackage{amsmath, amsthm, amssymb}
\usepackage{esint}
\usepackage{times}
\usepackage{graphicx}
\usepackage{marginnote}
\usepackage[unicode,breaklinks=true,colorlinks=true,linkcolor=black,urlcolor=black,citecolor=black]{hyperref}

\usepackage{tikz}

\ifnum\refcheck=1
\usepackage{refcheck}
\fi

\begin{document}
\def\linfeng#1{{\textcolor{red}{#1}}}
\def\bnew{\colr }
\def\enew{\colb {}}
\def\bold{\colu }
\def\eold{\colb {}}
\def\KK{(\diam \Omega)}
\def\KL{\kappa}
\def\KKnp{\diam \Omega}
\def\intb{[b]}
\def\PP{\mathcal{P}}
\def\dPP{\dot{\mathcal{P}}}
\def\LL{\mathcal L}
\def\qd{q_{\textrm D}^{}}
\def\AA{C_2}
\def\BB{(C_1+C_2)}
\def\pp{p}
\def\qq{q}
\def\MM{M}
\def\CC{C}
\def\uu{\tilde u }
\def\vv{\tilde v }
\def\ww{\tilde w }
\def\ggg{u}
\def\UU{\tilde U}
\def\eps{\epsilon}
\def\QQ{{\bar Q}}

\def\PP{P}
\def\RR{\mathbb R}
\def\TT{\mathbb T}
\def\ZZ{\mathbb Z}
\def\erf{\mathrm{Erf}}
\def\red#1{\textcolor{red}{#1}}
\def\blue#1{\textcolor{blue}{#1}}
\def\mgt#1{\textcolor{magenta}{#1}}

\def\les{\lesssim}
\def\ges{\gtrsim}

\renewcommand*{\Re}{\ensuremath{\mathrm{{\mathbb R}e\,}}}
\renewcommand*{\Im}{\ensuremath{\mathrm{{\mathbb I}m\,}}}

\ifnum\showllabel=1
\def\llabel#1{\marginnote{\color{lightgray}\rm\small(#1)}[-0.0cm]\notag}
\else
\def\llabel#1{\notag}
\fi

\newcommand{\norm}[1]{\left\|#1\right\|}
\newcommand{\nnorm}[1]{\lVert #1\rVert}
\newcommand{\abs}[1]{\left|#1\right|}
\newcommand{\NORM}[1]{|\!|\!| #1|\!|\!|}

\newtheorem{Theorem}{Theorem}[section]
\newtheorem{Corollary}[Theorem]{Corollary}
\newtheorem{Definition}[Theorem]{Definition}
\newtheorem{Proposition}[Theorem]{Proposition}
\newtheorem{Lemma}[Theorem]{Lemma}
\newtheorem{Remark}[Theorem]{Remark}

\def\theequation{\thesection.\arabic{equation}}
\numberwithin{equation}{section}

\definecolor{myblue}{rgb}{.8, .8, 1}

\newlength\mytemplen
\newsavebox\mytempbox

\makeatletter
\newcommand\mybluebox{%
\@ifnextchar[
{\@mybluebox}%
{\@mybluebox[0pt]}}

\def\@mybluebox[#1]{%
\@ifnextchar[
{\@@mybluebox[#1]}%
{\@@mybluebox[#1][0pt]}}

\def\@@mybluebox[#1][#2]#3{
\sbox\mytempbox{#3}%
\mytemplen\ht\mytempbox
\advance\mytemplen #1\relax
\ht\mytempbox\mytemplen
\mytemplen\dp\mytempbox
\advance\mytemplen #2\relax
\dp\mytempbox\mytemplen
\colorbox{myblue}{\hspace{1em}\usebox{\mytempbox}\hspace{1em}}}

\makeatother

\def\biglinem{\vskip0.5truecm\par==========================\par\vskip0.5truecm}
\def\inon#1{\hbox{\ \ \ \ \ }\hbox{#1}}                
\def\and{\text{\indeq and\indeq}}
\def\onon#1{\text{~~on~$#1$}}
\def\inin#1{\text{~~in~$#1$}}

\def\startnewsection#1#2{\section{#1}\label{#2}\setcounter{equation}{0}}   
\def\sgn{\mathop{\rm sgn\,}\nolimits}    
\def\Tr{\mathop{\rm Tr}\nolimits}    
\def\div{\mathop{\rm div}\nolimits}
\def\curl{\mathop{\rm curl}\nolimits}
\def\dist{\mathop{\rm dist}\nolimits}  
\def\sp{\mathop{\rm sp}\nolimits}  
\def\supp{\mathop{\rm supp}\nolimits}
\def\diam{\mathop{\rm diam}\nolimits}

\def\indeq{\quad{}}           
\def\colr{\color{red}}
\def\colb{\color{black}}
\def\coly{\color{lightgray}}

\definecolor{colorgggg}{rgb}{0.1,0.5,0.3}
\definecolor{colorllll}{rgb}{0.0,0.7,0.0}
\definecolor{colorhhhh}{rgb}{0.3,0.75,0.4}
\definecolor{colorpppp}{rgb}{0.7,0.0,0.2}
\definecolor{coloroooo}{rgb}{0.45,0.0,0.0}
\definecolor{colorqqqq}{rgb}{0.1,0.7,0}
\def\colg{\color{colorgggg}}
\def\collg{\color{colorllll}}
\def\cole{\color{coloroooo}}
\def\cole{\color{black}}
\def\coleo{\color{colorpppp}}
\def\colu{\color{blue}}
\def\colc{\color{colorhhhh}}
\def\colW{\colb}   
\definecolor{coloraaaa}{rgb}{0.6,0.6,0.6}
\def\colw{\color{coloraaaa}}

\def\comma{ {\rm ,\qquad{}} }            
\def\commaone{ {\rm ,\quad{}} }          
\def\nts#1{{\color{red}\hbox{\bf ~#1~}}}
\def\ntsf#1{\footnote{\color{colorgggg}\hbox{#1}}} 
\def\cmi#1{{\color{red}\hbox{IK: ~#1~}}}
\def\cml#1{{\color{red}\hbox{LL: ~#1~}}} 
\def\blackdot{{\color{red}{\hskip-.0truecm\rule[-1mm]{4mm}{4mm}\hskip.2truecm}}\hskip-.3truecm}
\def\bluedot{{\color{blue}{\hskip-.0truecm\rule[-1mm]{4mm}{4mm}\hskip.2truecm}}\hskip-.3truecm}
\def\purpledot{{\color{colorpppp}{\hskip-.0truecm\rule[-1mm]{4mm}{4mm}\hskip.2truecm}}\hskip-.3truecm}
\def\greendot{{\color{colorgggg}{\hskip-.0truecm\rule[-1mm]{4mm}{4mm}\hskip.2truecm}}\hskip-.3truecm}
\def\cyandot{{\color{cyan}{\hskip-.0truecm\rule[-1mm]{4mm}{4mm}\hskip.2truecm}}\hskip-.3truecm}
\def\reddot{{\color{red}{\hskip-.0truecm\rule[-1mm]{4mm}{4mm}\hskip.2truecm}}\hskip-.3truecm}

\def\tdot{{\color{green}{\hskip-.0truecm\rule[-.5mm]{2mm}{4mm}\hskip.2truecm}}\hskip-.2truecm}
\def\gdot{\greendot}
\def\bdot{\bluedot}
\def\ydot{\cyandot}
\def\rdot{\cyandot}

\def\fractext#1#2{{#1}/{#2}}

\def\quinn#1{\text{{\textcolor{colorqqqq}{#1}}}}
\def\igor#1{\footnote{\text{{\textcolor{colorqqqq}{#1}}}}}

%
%
%
%

\begin{abstract}
We provide an observability inequality in terms of a measurable set for general Gevrey regular functions. 
As an application, we establish an observability estimate from a measurable set for sums of Laplace eigenfunctions in a compact and connected boundaryless Riemannian manifold that belongs to the Gevrey class.
The estimate has an explicit dependence on the maximal eigenvalue.
\hfill \today
\end{abstract}

\maketitle

\startnewsection{Introduction}{sec01}	
In this paper, we establish an observability inequality from a measurable set for Gevrey-regular functions in a connected, bounded domain $\Omega\subset \mathbb{R}^d$, where $d\in\mathbb{N}$.  Under the assumptions that a doubling property holds for $f$ (see \eqref{EQ110} and \eqref{EQ130} below) and $f\in C^\infty (\Omega)$ is $\sigma$-Gevrey-regular, where $\sigma\geq 1$, we prove that for a measurable subset $E\subset \Omega$ with positive measure, we have
\begin{align}
	\Vert f\Vert_{L^\infty (\Omega )}
	\leq
	C \Vert f\Vert_{L^\infty (E)}
	,
	\label{EQ44}
\end{align}
for a constant $C>0$.  As a direct application, we provide an observability estimate from a measurable sets for the sum of Laplace eigenfunctions in compact and connected boundaryless Riemannian manifold that belongs to the $\sigma$-Gevrey class, for any $\sigma\geq 1$.  Additionally, we adapt the proof of the main theorem to give an observability estimate with the doubling property from \cite{IK1}, which applies to solutions of a one-dimensional parabolic equation of arbitrary order with Gevrey coefficients.
Observability estimates of type \eqref{EQ44} imply that certain parabolic equations with Gevrey coefficients are null-controllable on any subset of positive measure (see~\cite{AE, AEWZ, CSZ, EMZ}, where analyticity was assumed).

When the subset $E \subset \Omega$ in~\eqref{EQ44} is replaced by an open ball $B \subset \Omega$, such estimates are well-studied and are closely related to quantitative unique continuation properties (see, for example, the review papers \cite{Ke, V2}, and also \cite{CGT, IK1, IK2} for unique continuation properties for elliptic and parabolic equations with Gevrey coefficients). In~\cite{V1}, Vessella studied the recovery of an analytic function $f$ on a bounded, connected domain $\Omega \subset \mathbb{R}^d$ from its values on an open subset $E \subset \Omega$, establishing a stability estimate for $f$ that depends on the measure of $E$. Later, Apraiz and Escauriaza in~\cite{AE} provided an estimate of type \eqref{EQ44} for real-analytic functions when $E \subset \Omega$ is a measurable set with positive measure.

In~\cite{AEWZ}, the authors have proven observability estimates for the heat equation from subsets of positive measure or positive surface measure in $\Omega \times (0, T)$ and $\partial \Omega \times (0, T)$, respectively (see also \cite{N1, N2, PW, EMZ, LZ, LWYZ} for related estimates in other analytic settings).
In \cite{LoM}, Logunov and Malinnikova proved a quantitative propagation of smallness from a set of positive measure for the solution to an elliptic equation with Lipschitz coefficients in $\mathbb{R}^n$.
Using the doubling index of solutions to elliptic equations, they proved an observability estimate from a measurable set for the Laplace eigenfunction on a $C^\infty$ smooth Riemannian manifold;
see also \cite{DYZ,BM} for other observability estimates from measurable sets of solutions to parabolic equations.
Recently, Y.~Zhu \cite{Z} studied quantitative propagation of smallness for solutions to elliptic equations in two dimensions. Note that elliptic equations in the plane are special, as their solutions can be reduced to holomorphic functions using quasi-regular mappings and representation theorems. A key idea used in~\cite{V1, AE, Z} is based on the Hadamard three-circle theorem and Cauchy estimates for holomorphic functions, tools that are unavailable in the Gevrey setting.

This paper establishes an observability estimate from a measurable sets for general Gevrey functions satisfying a doubling property. 
As a direct application, we present an observability estimate for sums of Laplace eigenfunctions in a compact and connected Riemannian manifold that has Gevrey regularity.
Notably, our method is adaptable to provide observability estimates from a measurable sets in other settings. For instance, we consider~\cite{IK1}, where the authors gave a quantitative unique continuation estimate for a one-dimensional parabolic equation with Gevrey coefficients, and~\cite{IK2}, where a higher-order elliptic equation with Gevrey coefficients was studied, yielding an observability estimate for unique continuation with simple complex characteristics and only Gevrey coefficients.

Our results are obtained using the elliptic iterate theorem from~\cite{LM1, LM2}, which addresses the interior regularity of solutions to elliptic problems with Gevrey-class coefficients. 
Another essential component is the approximation of a function by a polynomial with error estimates based on Gevrey regularity (see~\cite{IK1, CKL}). This paper also draws from \cite{AE} in finding separation points within a measurable set with positive measure, as detailed in the proof of Theorem \ref{T01} below.

The paper is organized as follows. In Section~\ref{sec02}, we state the main result, Theorem~\ref{T01}, an observability inequality for Gevrey-regular functions with a doubling property. As an application, Theorem~\ref{T02} provides an observability estimate for the sum of Laplace eigenfunctions with explicit dependence on the largest eigenvalue and number of eigenfunctions. 
In Section~\ref{sec03}, we prove all main results.

\startnewsection{Main results}{sec02}
Let $\Omega$ be a connected, bounded domain in $\mathbb{R}^d$,
where $d\in\mathbb{N}$ is fixed.
For a function
$f\in C^\infty (\Omega)$, 
assume that there exist constants $M\geq 1$ and $\delta>0$ such that
\begin{align}
	\Vert \partial^\alpha f  \Vert_{L^\infty (\Omega)}
	\leq
	\frac{M |\alpha|!^\sigma}{\delta^{|\alpha|}}
	\Vert f\Vert_{L^\infty (\Omega)}
	\comma
	\alpha \in \mathbb{N}_0^d
	,
	\label{EQ110}
\end{align}
where $\sigma\geq 1$ is the Gevrey parameter.
Additionally, we assume a doubling-type property, namely that there exist constants $\kappa\geq 2$ and $r_0>0$ such that
\begin{align}
	\Vert f\Vert_{L^\infty (B_{2r} (x) \cap \Omega)}
	\leq
	\KL
	\Vert f\Vert_{L^\infty (B_r (x) \cap \Omega)}
	\comma
	r\in (0,r_0]
	\comma
	x\in \Omega
	.
	\label{EQ130}
\end{align}
The condition~\eqref{EQ130} is known to hold for eigenfunctions of the Laplace operator (see~\cite{DF1,DF2}). 
There are also other variants of the doubling property of type
\eqref{EQ130}, and, in fact, the field
of quantitative unique continuation is concerned with estimates of type~\eqref{EQ130} (see~\cite{IK1, IK2, Ku1}).

The following theorem provides an observability estimate
from a measurable set
for functions~$f$ satisfying \eqref{EQ110} and~\eqref{EQ130}.

\cole
\begin{Theorem}
	\label{T01}
Suppose that $\Omega$ is a $C^1$ domain.
Let $M,\delta, \KL,r_0>0$ be constants,
let $\sigma\geq 1$,
and assume that 
	$f\in C^\infty (\Omega)$ satisfies \eqref{EQ110} and~\eqref{EQ130}. 
Then, for any measurable set $E\subset \Omega$ with positive measure,
we have
	\begin{align}
		\Vert f\Vert_{L^\infty (\Omega )}
		\leq
		C 
		\Vert f\Vert_{L^\infty (E)}
		,
		\label{EQ100}
	\end{align}
where $C=C(M,\delta,\KL, r_0, \sigma, |\Omega|/|E|, \Omega)>0$ is a constant. 
\end{Theorem}
\colb

From the proof, we may deduce a more precise information on the
size of the constant in~\eqref{EQ100}. Thus, when $\sigma=1$, we obtain
 \begin{align}
	\begin{split}
		\Vert f\Vert_{L^\infty (\Omega)}
		   \leq
		M \KL^C
		(|\Omega|/|E|)^{C\log \KL}
		\Vert f\Vert_{L^\infty (E)}
		,
	\end{split}
	\label{EQ301}
\end{align}
for a constant $C\geq1$ depending only on $\delta, r_0$, and $\Omega$,
while if $\sigma>1$, then
  \begin{align}
	\begin{split}
		\Vert f\Vert_{L^\infty (\Omega)}
		&\leq
		 M
		 \KL^{C }
		(|\Omega|/ |E|)^{C\max\{\log \KL, B\}}
		\max\{\log\KL, B\}^{C(\sigma-1) \log \KL}
		\Vert f\Vert_{L^\infty (E)},
	\end{split}
	\label{EQ305}
\end{align}
where $B,C\geq1$ are constants depending only on $\delta, r_0, \sigma$, and $\Omega$; we have assumed that $\kappa\geq2$,
without loss of generality.

As an application, we consider sums of eigenfunctions of the Laplace operator on a Gevrey manifold.
Let $\mathcal{M}\subseteq \mathbb{R}^k$ be a $d$-dimensional smooth compact and connected Riemannian manifold with metric $g$, where $k\in \mathbb{N}$.
Fix $\sigma\geq1$, and
let $\mathcal{M}$ be a $\sigma$-Gevrey regular manifold without boundary with a metric that is $\sigma$-Gevrey.
Denote by $\Delta_{\mathcal{M}}$ the Laplace-Beltrami operator on $\mathcal{M}$ and let $\phi_i$ be an eigenfunction of $-\Delta_{\mathcal{M}}$ with the distinct eigenvalue $\lambda_i\geq 1$, i.e., 
\begin{align}
	&
	-\Delta_{\mathcal{M}}
	\phi_i
	= \lambda_i 
	\phi_i
	\inon{in~$\mathcal{M}$}
	,
	\label{EQ24}
\end{align}
for $i=1,\ldots,m$, where $m\in \mathbb{N}$. 
Denote by $h= \sum_{i=1}^m \phi_i$. 
The following theorem provides an observability estimate from a measurable set for the sum of Laplace eigenfunctions satisfying \eqref{EQ24}.

\begin{Theorem}
\label{T02}
Let $\sigma \geq 1$.
Assume that $(\mathcal{M},g)$ is a compact and connected Riemannian manifold that belongs to the $\sigma$-Gevrey class.
Then there exists a constant $C\geq 1$ such that for every measurable subset $E\subset \mathcal{M}$ with positive measure we have
	\begin{align}
		\Vert h\Vert_{L^\infty (\mathcal{M})}
		\leq
		\left(\frac{C}{|E|} \right)^{C\gamma}
		{\gamma}^{C (\sigma-1) \gamma}
		\Vert h\Vert_{L^\infty (E)}
		,
		\llabel{EQ71}
	\end{align}
where
$\lambda=\max \{\lambda_1, \ldots, \lambda_m\}$ and $\gamma= C\sqrt{\lambda}+ Cm^2 \log m +C$.
\end{Theorem}
\colb

We emphasize that both $m$ and $\lambda$ are allowed to be arbitrarily large and the constant $C$ is independent of $E$, $m$, and~$\lambda$.

The statement and the proof of Theorem~\ref{T01} can be adapted to
prove an observability estimate from a measurable set with other types
of quantitative unique continuation properties.
Here, suppose that there exist constants $a,b> 0$ and $r_0>0$ such that 
\begin{align}
	\Vert f\Vert_{L^\infty (\Omega )}
	\leq
	\exp \left(
	\frac{a}{r^b}
	\right)
	\Vert f\Vert_{L^\infty (B_r (x) \cap \Omega)}
	\comma
	r\in (0,r_0]
	\comma
	x\in \Omega.
	\label{EQ030}
\end{align}
The assumption \eqref{EQ030} is known for solutions of a one-dimensional parabolic equation of arbitrary order with Gevrey coefficients; see~\cite{IK1} for the motivation and other related equations.

The following theorem provides an observability estimate from a measurable set for functions~$f$ satisfying \eqref{EQ110} and~\eqref{EQ030}.

\cole
\begin{Theorem}
	\label{T03}
Suppose that $\Omega$ is a $C^1$ domain, and assume that
$f\in C^\infty (\Omega)$ satisfies \eqref{EQ110} and~\eqref{EQ030},
where
$M,\delta, a,b,r_0>0$ are constants such that
$1\leq \sigma < 1+1/b$.
Then, for any measurable set $E\subset \Omega$ with  positive measure, we 
have
	\begin{align}
		\Vert f\Vert_{L^\infty (\Omega )}
		\leq
		C 
		\Vert f\Vert_{L^\infty (E)}
		,
		\llabel{EQ80}
	\end{align}
	where $C=C(M,\delta,a,b, r_0, \sigma, |\Omega|/|E|, \Omega)>0$ is a constant.
\end{Theorem}
\colb
Throughout this paper, we use $C\geq1$ to denote a sufficiently large constant that may depend on $d$ and $\Omega$ without mention; the value of $C$ may vary from line to line.
All three theorems are proven in the next section.

\startnewsection{Proof of the main results}{sec03}
First, we prove the following lemma, which estimates the number of balls needed to cover the domain~$\Omega$.

\cole
\begin{Lemma}
\label{L00}
	Let $\Omega$ be a bounded domain in $\mathbb{R}^d$.
	There exists a constant $C>0$ such that
	 for any $r>0$ we have an open cover
	\begin{align}
		\Omega
		\subset
		\cup_{i=1}^{([C/r]+1)^d}
		B_{r} (x_i),
		\llabel{EQ400}
	\end{align}
where $x_i\in \Omega$ for $i=1,\ldots, ([C/r]+1)^d$. 
\end{Lemma}
\colb

Since the proof is straightforward, we omit it.

\begin{proof}[Proof of Theorem~\ref{T01}]
First, we prove an observability estimate for a ball $B_r (\hat{x})$, where $\hat{x}\in \Omega$ and $r\in (0,r_0]$.	
Let $\bar{x} \in \Omega$ be such that $|f(\bar{x})|\geq \sup_{\Omega} |f|/2$. 
Since $\Omega$ is connected, we may join $\bar{x}$ and $\hat{x}$ by an overlapping chain of balls with radius $\hat{r}=r2^{[\log_2 (r_0/r)]}$, where we use $[\cdot]$ to denote the integer part.
Without loss of generality, we may assume that $r_0\leq 1$.
Using \eqref{EQ130}, we obtain
\begin{align}
	|f(\bar{x})|
	\leq
	\Vert f\Vert_{L^\infty (B_{\hat{r}} (\bar{x}) \cap \Omega)}
	\leq 
	\KL^{C/ \hat{r}}
	\Vert f\Vert_{L^\infty (B_{\hat{r}} (\hat{x}) \cap \Omega )}
	\leq 
	\KL^{C}
	\Vert f\Vert_{L^\infty (B_{\hat{r}} (\hat{x}) \cap \Omega )},
	\label{EQ170}
\end{align}
where $C>0$ is a constant depending on $r_0$.
We use \eqref{EQ130} 
on concentric balls centered at $\hat{x}$ to obtain
\begin{align}
	\Vert f\Vert_{L^\infty (B_{\hat{r}} (\hat{x})\cap \Omega)}
	\leq
	\KL^{[\log_2 (r_0/r)]}	
	\Vert f\Vert_{L^\infty (B_r (\hat{x})\cap \Omega)}
	.
	\label{EQ172}
\end{align}
Combining \eqref{EQ170} and \eqref{EQ172}, we arrive at
\begin{align}
	\begin{split}
		\Vert f\Vert_{L^\infty (\Omega)}
		&
		\leq
		2 |f(\bar{x})|
		\leq
		2 
		\KL^{C+[\log_2 (r_0/r)]}
		\Vert f\Vert_{L^\infty (B_{r} (\hat{x})\cap \Omega)}
		\leq
		\frac{
			\KL^{C}}{r^{\log_2 \KL}}
		\Vert f\Vert_{L^\infty (B_{r} (\hat{x})\cap \Omega)},
		\label{EQ163}	
	\end{split}
\end{align}
for any $r\in (0,r_0]$.	
	
Let $\epsilon=1/100$.
For each $z_\alpha \in \partial \Omega$, there exists a local coordinate system $(z', z_d) \in \mathbb{R}^{d-1}\times \mathbb{R}$, a constant $R_\alpha>0$, and a $C^1$ function $\phi_\alpha: \mathbb{R}^{d-1}\to \mathbb{R}$ such that:
\begin{enumerate}[label=(\roman*)]
	\item $B_{R_\alpha} (z_\alpha) \cap \Omega =\{ (z',z_d)\in B_{R_\alpha} (z_\alpha): z_d > \phi_\alpha (z')\}$,
	\item $B_{R_\alpha} (z_\alpha) \cap \partial \Omega =\{ (z',z_d)\in B_{R_\alpha} (z_\alpha): z_d = \phi_\alpha (z')\}$, and
	\item $|\phi_\alpha (z_1')- \phi_\alpha (z_2') |\leq \epsilon |z_1'-z_2'|$, for any $(z_1', \phi_\alpha (z_1')), (z_2', \phi_\alpha (z_2')) \in B_{R_\alpha} (z_\alpha)$.
\end{enumerate}
It is clear that $\partial \Omega\subset \cup_{z_\alpha \in \partial \Omega} B_{R_\alpha/2} (z_\alpha)$.
Since $\partial \Omega$ is compact, there exists a finite open cover 
\begin{align}
	\partial \Omega \subset \cup_{j=1}^N B_{R_j/2} (z_j).
	\label{EQ401}
\end{align}
Denote $R_0=\min\{R_1,\ldots,R_N\}/4$ and $\tilde{r}_0=\min \{R_0, r_0 \}$.
Let $r\in (0, \tilde{r}_0]$ be a small constant, to be determined below.
Using Lemma~\ref{L00}, there exists an open cover
\begin{align}
	\Omega
	\subset
	\cup_{j=1}^{([C/r]+1)^d}
	B_{r} (y_j),
	\llabel{EQ400}
\end{align}
where $y_j\in \Omega$ for $j=1,\ldots, ([C/r]+1)^d$.
Since $E\subset \Omega$, there exists a ball $B_r (x)$ with $x\in \Omega$ such that
\begin{align}
	|B_r (x) \cap E |
	\geq 
	\frac{|E|}{([C/r]+1)^d}
	\geq 
	\frac{|E|}{(C/r)^d}
	=
	\frac{|E| r^d}{C}
	.
	\label{EQ40}
\end{align}

First, consider the case when $B_r(x) \cap \partial \Omega \neq \emptyset$. 
Let $z\in B_r(x)\cap \partial\Omega$. From \eqref{EQ401}, there exists $j\in \{1,\ldots, N\}$ such that $z\in B_{R_j/2} (z_j)$.
For any $y\in B_r(x)$, we have
\begin{align}
	|z_j-y|
	\leq
	|z_j-z|+|z-x|+|x-y|
	<
	\frac{R_j}{2}
	+
	r
	+
	r
	\leq
	R_j,
	\llabel{EQ402}
\end{align}
where we used $r\leq \min\{R_1,\ldots,R_N\}/4$ in the last inequality. 
Thus, we conclude that $B_r (x)\subset B_{R_j} (z_j)$, which leads to $B_r (x)\cap \Omega\subset B_{R_j} (z_j) \cap \Omega$ and $B_r (x)\cap \partial\Omega\subset B_{R_j} (z_j) \cap \partial\Omega$.
It is clear that $B_{r/10} (x+r e_d/2)\subset \Omega$, where $e_d$ is the vertical unit vector in the local coordinate system. 
Denote $\tilde{x}=x+r e_d/2$.
Since the graph of $B_{R_j} (z_j) \cap \partial\Omega$ in the local coordinate system has a gradient at most $\epsilon=1/100$, we follow the arguments in \cite{KN} to conclude that each ray starting from any point in $B_{r/10} (\tilde{x})$ intersects $B_{r} (x)\cap \partial \Omega$ at most once.

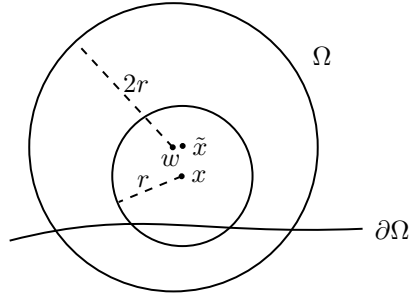
\begin{center}
	\tikzset{every picture/.style={line width=0.75pt}} 
	
	\begin{tikzpicture}[x=0.75pt,y=0.75pt,yscale=-0.5,xscale=0.5]
		
		\draw    (126.5,240) .. controls (258.5,206) and (311.5,235) .. (482.5,228) ;
		\draw   (230,174.75) .. controls (230,135.68) and (261.68,104) .. (300.75,104) .. controls (339.82,104) and (371.5,135.68) .. (371.5,174.75) .. controls (371.5,213.82) and (339.82,245.5) .. (300.75,245.5) .. controls (261.68,245.5) and (230,213.82) .. (230,174.75) -- cycle ;
		\draw  [dashed]  (300, 175) -- (236.5,201) ;
		\draw   (146.25,145.75) .. controls (146.25,65.39) and (211.39,0.25) .. (291.75,0.25) .. controls (372.11,0.25) and (437.25,65.39) .. (437.25,145.75) .. controls (437.25,226.11) and (372.11,291.25) .. (291.75,291.25) .. controls (211.39,291.25) and (146.25,226.11) .. (146.25,145.75) -- cycle ;
		\draw  [dashed]  (291,145.75) -- (193.5,40) ;
		
		\draw (252,171.4) node [anchor=north west][inner sep=0.75pt]    {$r$};
		\draw (238,69.4) node [anchor=north west][inner sep=0.75pt]    {$2r$};
		\draw (276,150) node [anchor=north west][inner sep=0.75pt]    {$w$};
		\draw (429,41.4) node [anchor=north west][inner sep=0.75pt]    {$\Omega $};
		\draw (492,216.4) node [anchor=north west][inner sep=0.75pt]    {$\partial \Omega $};
		
		\filldraw[black] (301,144.4) circle (1.5pt) node[anchor=west]{$\tilde{x}$};
		
		\filldraw[black] (291,145.75) circle (1.5pt) node[anchor=west]{};
		
			\filldraw[black] (300, 175) circle (1.5pt) node[anchor=west]{$x$};
	\end{tikzpicture}
\end{center}
Let $w\in B_{r/10} (\tilde{x})$ be such that $|f(w)|\geq \sup_{B_{r/10} (\tilde{x})} |f|/2$. 
It is clear that $B_r(x)\subset B_{2r} (w)$.
Using the spherical coordinates centered at $w$ with radius $2r$, we have
\begin{align}
|B_r (x) \cap E |
\leq
Cr^{d-1}
\int_{S^{d-1}}
\bigl|\{t\in [0,2r]: w+t \mu \in B_r(x)\cap E\}\bigr|
\, d\mu
.
\label{EQ41}
\end{align}
From \eqref{EQ40} and \eqref{EQ41}, it follows that
there exists $\mu_0 \in S^{d-1}$ such that
\begin{align}
\bigl| \{t\in [0,2r]: w+t\mu_0
\in B_r(x)\cap E\} \bigr|
\geq
\frac{|B_r (x)\cap E|}{Cr^{d-1}}
\geq
\frac{|E| r}{C}
.
\label{EQ14}
\end{align}
Note that each ray that starts from $w\in B_{r/10} (\tilde{x})$ intersects $B_{r} (x)\cap \partial\Omega$ at most once. 
Consequently, we deduce that $L:=\{w+t\mu_0: t\in [0, 2r] \} \cap B_r (x) \cap \Omega$ is an interval. From \eqref{EQ14}, we infer that
\begin{align}
	|L\cap E|
	=
	\bigl| \{t\in [0,2r]: w+t\mu_0
	\in B_r(x)\cap E\} \bigr|
	\geq
	\frac{|E| r}{C}.
	\label{EQ502}
\end{align}
With $n\in \mathbb{N}_0$,
choose $n+1$ points $\{x_0,x_1, \ldots, x_n\} \subset \overline{L\cap E}$ such that 
\begin{align}
	|x_i - x_{i-1}| 
	\geq 
	\frac{ |L\cap E |}{n+1}
	\geq 
	\frac{|E|r}{C(n+1)}
	\comma
	i=1,\ldots,n,
	\llabel{EQ15}
\end{align}
where we used \eqref{EQ502}.
To see that we can properly separate the points, we first rotate the interval $L$ and then set $x_0= \inf (\overline{L\cap E})$ and
\begin{align}
	x_i
	=
	\inf
	\left\{
	\overline{L\cap E}
	\cap
	\left[x_{i-1}
	+ \frac{|L\cap E|}{n+1}
	,\infty\right)
	\right\}
	\comma
	i=1,\ldots,n.
	\llabel{EQ108}
\end{align}
Let 
\begin{align}
P(t)
=
\sum_{i=0}^n
f(x_i)
\frac{\prod_{j\neq i} (t-x_j)}{\prod_{j\neq i} (x_i - x_j)}
\llabel{EQ19}
\end{align}
be the Hermite interpolation polynomial defined on the interval $L$,
which satisfies $P(x_i)=f(x_i)$, for $i=0,1,\ldots,n$.
It is clear that $|L|\leq 2r$, so it follows
that
\begin{align}
\begin{split}
\Vert P\Vert_{L^\infty (L)}
\leq
\Vert f\Vert_{L^\infty (E)}
\frac{(2r)^n }{ ( |E| r)^n}
\sum_{i=0}^n
\frac{ (C (n+1) )^{n}}{i!(n-i)!}
\leq
\frac{C^{n+1} \Vert f\Vert_{L^\infty (E)}
}{|E|^n }
,
\label{EQ20}	
\end{split}
\end{align}
where we used
  \begin{equation}
    \biggl|   \prod_{j\neq i} (x_i - x_j)\biggr|
    \geq
    i!(n-i)!
    \left(
      \frac{|E| r}{C (n+1)}
    \right)^{n}
    \comma
    i=0,1,\ldots,n
   \label{EQ04}
  \end{equation}
in the first inequality and 
the Stirling formula in the second. 
By the Hermite interpolation theorem, for any $t\in L$ there exists some $s \in L$ such that
\begin{align}
f(t)- P(t)
=
\frac{(t-x_0) (t-x_1) \cdots (t-x_n)}{(n+1)!}
f^{(n+1)} (s)
.
\label{EQ10}
\end{align}
Note that $f$ can be considered as a function of a single variable defined on $L$, and we thus use $f^{(m)}$ to denote the $m$-th derivative of $f$.
Thus, from \eqref{EQ110} and \eqref{EQ10}, it follows that
\begin{align}
\begin{split}
	\Vert f- P\Vert_{L^\infty (L)}
	&
	\leq
	\frac{(2r)^{n+1} \Vert f^{(n+1)} \Vert_{L^\infty (L)} }{(n+1)!}
	\leq	
	\frac{	M
	C^{n+1}  r^{n+1} (n+1)!^{\sigma-1} 		\Vert f\Vert_{L^\infty (\Omega)}}{\delta^{n+1}}	
	.
	\label{EQ21}
\end{split}
\end{align}
Now, we appeal to \eqref{EQ163}, which yields
\begin{align}
	\begin{split}
		\Vert f\Vert_{L^\infty (\Omega)}
	&
	\leq
	\KL^{C} (r/10)^{-\log_2 \KL}
	\Vert f\Vert_{L^\infty (B_{r/10} (\tilde{x}))}
	\leq
	2
	\KL^{C} r^{-\log_2 \KL}
	|f(w)|
	\leq
	\KL^{C} r^{-\log_2 \KL}
	\Vert f\Vert_{L^\infty (L)}
	,
	\label{EQ150}
	\end{split}
\end{align}
since $\KL\geq 2$.
From \eqref{EQ20} and \eqref{EQ21}--\eqref{EQ150}, it follows that
\begin{align}
\begin{split}
	\Vert f\Vert_{L^\infty (\Omega)}
	&
	\leq
	\KL^{C} r^{-\log_2 \KL}
(	\Vert P\Vert_{L}
	+
	\Vert f-P\Vert_{L})
	\\&
	\leq
	\KL^{C} r^{-\log_2 \KL}
	\left(
	\frac{ C^{n+1}	\Vert f\Vert_{L^\infty (E)}
	}{|E|^n }
	+
	\frac{	M
	C^{n+1}
	r^{n+1} (n+1)^{(n+1)(\sigma-1)} 
	\Vert f\Vert_{L^\infty (\Omega)}}{\delta^{n+1}}
	\right)
	\\&
	\leq
	\KL^{C} r^{-\log_2 \KL}
	C^{n+1}
	\bigl(
	(|\Omega|/|E|)^{n} \Vert f\Vert_{L^\infty (E)}
	+
	M (1/\delta)^{n+1}  r^{n+1}
	(n+1)^{(n+1)(\sigma-1)}
	\Vert f\Vert_{L^\infty (\Omega)}
	\bigr).
	\label{EQ13}
\end{split}
\end{align}
Note that \eqref{EQ13} holds for any $n\in \mathbb{N}_0$ and $r\in (0,\tilde{r}_0]$. 

First, we consider the case $\sigma=1$, when we set
\begin{align}
	r=
	\tilde{r}_0
	\left(
	\frac{\Vert f\Vert_{L^\infty (E)}}{M \Vert f\Vert_{L^\infty (\Omega)}}
	\right)^{1/(n+1)}.
	\label{EQ103}
\end{align}
It is clear that $r\leq \tilde{r}_0$ in \eqref{EQ103}, for any $n\in \mathbb{N}_0$.
Inserting $n=n_0=2[\log_2 \KL]+2$ and \eqref{EQ103} into the far-right side of \eqref{EQ13}, we arrive at
\begin{align}
	\begin{split}
		\Vert f\Vert_{L^\infty (\Omega)}
		&
		\leq
		\KL^C
		\tilde{r}_0^{-\log_2 \KL}
		C^{n_0+1}
		\Vert f\Vert_{L^\infty (E)}
		\bigl( (|\Omega|/|E|)^{n_0} 
		+
		(1/\delta)^{n_0+1}
		\tilde{r}_0^{n_0+1}
		\bigr) 
		\left(
		\frac{M \Vert f\Vert_{L^\infty (\Omega)}}{\Vert f\Vert_{L^\infty (E)}}
		\right)^{(\log_2 \KL)/(n_0+1)}
		\\&
		\leq
		\KL^C (|\Omega|/|E|)^{2[\log_2 \KL]+2} 
		C^{[\log_2 \KL]+1}
		M^{\gamma}
		\Vert f\Vert_{L^\infty (\Omega)}^\gamma
		\Vert f\Vert_{L^\infty (E)}^{1-\gamma},
		\llabel{EQ104}
	\end{split}
\end{align} 
where $\gamma=(\log_2 \KL)/(2[\log_2 \KL] +3)$.
Note that $0<\gamma\leq 1/2$.
Thus, we get
  \begin{align}
  \begin{split}
	\Vert f\Vert_{L^\infty (\Omega)}
		&
   \leq
   \big(M^\gamma \KL^C
   (|\Omega|/|E|)^{2[\log_2 \KL]+2}
   \big)^{1/(1-\gamma)}
   \Vert f\Vert_{L^\infty (E)}
   \leq
   M
   \KL^C
   (|\Omega|/|E|)^{C\log \KL}
   \Vert f\Vert_{L^\infty (E)}
   ,
  \end{split}
   \label{EQ01}
  \end{align}
completing the proof in the case
$\sigma=1$.

When $\sigma>1$, we set
\begin{align}
	r= 
	\left(
	\frac{ \delta^{n+1}\Vert f\Vert_{L^\infty (E)}}{ M (n+1)^{(n+1) (\sigma-1)} \Vert f\Vert_{L^\infty (\Omega)}} \right)^{1/(n+1)}
	\label{EQ30},
\end{align}
where $n\in \mathbb{N}_0$ is a sufficiently large integer to be determined below;
in particular, we need to choose it so that the right-hand side of~\eqref{EQ30}
is less than or equal to~$\tilde{r}_0$.
Inserting \eqref{EQ30} into the far-right side of \eqref{EQ13}, we obtain
\begin{align}
\begin{split}
	\Vert f\Vert_{L^\infty (\Omega)}
	&
	\leq
	\KL^C C^{n+1}
	( (|\Omega|/|E|)^n+1)
	\Vert f\Vert_{L^\infty (E)}
	\left(
	\frac{ M (n+1)^{(n+1) (\sigma-1)} \Vert f\Vert_{L^\infty (\Omega)}}{ \delta^{n+1}\Vert f\Vert_{L^\infty (E)}} \right)^{(\log_2 \KL)/(n+1)}
	\\&
	\leq
	\KL^C
	C^{n+1} 
	 (|\Omega|/|E|)^{n} (n+1)^{(\sigma-1) \log_2 \KL }
	 \\&\indeq\indeq
	 \times
	 M^{(\log_2 \KL)/(n+1)}
	 (1/\delta)^{\log_2 \KL}
	\Vert f\Vert_{L^\infty (\Omega)}^{(\log_2 \KL)/(n+1)}
	\Vert f\Vert_{L^\infty (E)}^{1-(\log_2 \KL) /(n+1)}
	.
\label{EQ34}
\end{split}
\end{align}
Let 
\begin{align}
		n_1=
	2\bigl[\max\bigl\{\log_2 \KL, (\delta \tilde{r}_0^{-1})^{1/(\sigma-1)}\bigr\}\bigr]+1.
	\llabel{EQ151}
\end{align}
Note that this choice of $n_1$ ensures that $r\leq \tilde{r}_0$ in \eqref{EQ30} and $\eta=(\log_2 \KL)/(n_1+1) \in (0,1/2]$.
Inserting $n=n_1$ into~\eqref{EQ34}, we conclude that
\begin{align}
	\begin{split}
	\Vert f\Vert_{L^\infty (\Omega)}
	\leq
	\KL^C C^{n_1+1} 
	(|\Omega|/ |E|)^{n_1}
	 (n_1+1)^{(\sigma-1) \log_2 \KL}
	M^{\eta}	
	\Vert f\Vert_{L^\infty (\Omega)}^{\eta}
	\Vert f\Vert_{L^\infty (E)}^{1-\eta}
	.
	\llabel{EQ105}
	\end{split}
\end{align}
Therefore, we obtain
  \begin{align}
  \begin{split}
   \Vert f\Vert_{L^\infty (\Omega)}
    &\leq
     (	\KL^C C^{n_1+1} 
          (|\Omega|/ |E|)^{n_1}
     M^{\eta}	
     (n_1+1)^{(\sigma-1) \log_2 \KL}
       )^{1/(1-\eta)}
    \Vert f\Vert_{L^\infty (E)}
    \\
    &
    \leq
    M	\KL^{C} 
    (|\Omega|/ |E|)^{C\max\{\log \KL, B\}}
    \max\{\log \KL, B\}^{C(\sigma-1) \log \KL }
    \Vert f\Vert_{L^\infty (E)}
    ,
  \end{split}
   \label{EQ05}
  \end{align}
where $B=(\delta \tilde{r}_0^{-1})^{1/(\sigma-1)}$.

From \eqref{EQ34} and \eqref{EQ05}, we conclude the proof of the theorem in the case when $B_r (x)\cap \partial \Omega\neq \emptyset$.
On the other hand, if $B_r (x)\cap \partial \Omega= \emptyset$, it follows that $B_r(x)\subset \Omega$.
Let $w\in B_r (x)$ be such that $|f(w)|\geq \sup_{B_r (x)} |f|/2$. 
Proceeding analogously to \eqref{EQ41}--\eqref{EQ05}, we complete the proof.
\end{proof}

Next, we aim to prove Theorem~\ref{T02}.
Consider a second-order differential operator
\begin{align}
	\mathcal{A}
	u
	=
	\sum_{|\alpha|\leq 2}
	a_\alpha (x)
	D^\alpha u,
	\label{EQ74}
\end{align}
where $a_\alpha(x)$ is defined on a bounded open set $\Omega\subset \mathbb{R}^d$.
We adopt the usual multi-index notation that
$D^\alpha= D_{x_1}^{\alpha_1} \cdots D_{x_d}^{\alpha_d}$ for $\alpha=(\alpha_1,\ldots,\alpha_d)\in \mathbb{N}_0^d$ and $|\alpha|=\alpha_1+\cdots+ \alpha_d$. Let $\sigma\geq 1$.
We shall state general assumptions which make $\mathcal{A}$ a $\sigma$-Gevrey-regular elliptic operator. 
Note that these restrictions shall be used only implicitly (they are the hypotheses needed for \cite[Theorem VIII.2.4]{LM2}).

\textit{Assumptions on the Gevrey regularity of $\mathcal{A}$}:
\begin{enumerate}[label=(\roman*)]
	\item The coefficients $a_\alpha$ of $\mathcal{A}$ are $\sigma$-Gevrey on a neighborhood of $\bar{\Omega}$;
	
	
	
	\item $\mathcal{A}$ is properly elliptic in $\bar{\Omega}$.
	
%
\end{enumerate}

We recall the interior elliptic iterates theorem due to Lions and Magenes.
\cole
\begin{Lemma}[\cite{LM2}]
	\label{L02}
	Let $\sigma\geq 1$ and $\Omega$ be an arbitrary bounded domain in $\mathbb{R}^d$.
	Suppose that $u\in C^\infty (\bar{\Omega})$ and there exist constants $M_1, M_2>0$ (depending on $u$) such that
	\begin{align}
		\Vert \mathcal{A}^n u\Vert_{L^2 (\Omega)}
		\leq
		M_1 M_2^n (2n)!^\sigma
		\comma
		n \in \mathbb{N}_0
		.
	\end{align}
	Then for any compact subset $K \subset \Omega$, we have
	\begin{align}
		\Vert D^\alpha u \Vert_{L^\infty (K)} 
		\leq
		\bar{M}_1 \bar{M}_2^{|\alpha|}
		|\alpha|!^\sigma 
		\comma
		\alpha \in \mathbb{N}_0^d
		,
	\end{align}
	where $\bar{M}_1, \bar{M}_2>0$ are constants.
\end{Lemma}
\colb
For the proof, see \cite[p.55, Theorem~VIII.2.4]{LM2}

\begin{proof}[Proof of Theorem~\ref{T02}]
Denote by $B^g$ a geodesic ball in $\mathcal{M}$ with the metric $g$. 
For each $p_\alpha\in \mathcal{M}$ there exists $B^g_{R_\alpha} (p_\alpha)$ such that in a local coordinate system $x_1, \ldots, x_d$, we have
\begin{align}
	\Delta_{\mathcal{M}} 
	= 
	\frac{1}{\sqrt{g}}
	\frac{\partial}{\partial x_i}
	\left(
	\sqrt{g} g^{ij} \frac{\partial}{\partial x_j}
	\right)
	=
	\sum_{i,j=1}^d
	b_{ij} \frac{\partial^2}{\partial x_i \partial x_j}
	+
	\sum_{i=1}^d b_i
	\frac{\partial}{\partial x_i}.
	\label{EQ702}
\end{align}
We denote the operator on the far right side of \eqref{EQ702} by~$\tilde{\mathcal{A}}$.
By the assumption on the metric $g$, it follows that $\tilde{\mathcal{A}}$ is a $\sigma$-Gevrey regular elliptic operator.
It is clear that $\mathcal{M}\subset \cup_{p_\alpha\in \mathcal{M}}
B_{R_\alpha/2}^g (p_\alpha)$.
Since $\mathcal{M}$ is compact, there exists a finite open cover 
\begin{align}
	\mathcal{M}
	\subset
	\cup_{i=1}^N
	B_{R_i/2}^g (p_i)
	,
	\llabel{EQ703}
\end{align}
where $p_i\in \mathcal{M}$ for $i=1,\ldots,N$. 
Since $E\subset \mathcal{M}$, there exists a ball $B_{R/2}^g (p)$ with $p\in \mathcal{M}$ such that
\begin{align}
	|B_{R/2}^g (p) \cap E|
	\geq
	\frac{|E|}{N}.
	\label{EQ701}
\end{align}
Let $n\in \mathbb{N}_0$.
From \eqref{EQ24}, it follows that
\begin{align}
	\begin{split}
		\Delta_{\mathcal{M}}^n h
		&
		= 
		\sum_{i=1}^m 
		\Delta_{\mathcal{M}}^n \phi_i
		=
		\sum_{i=1}^m (-1)^n \lambda_i^n \phi_i
		,
		\llabel{EQ84}
	\end{split}
\end{align}
from where 
\begin{align}
	\begin{split}
		\Vert \Delta_{\mathcal{M}}^n h \Vert_{L^2 (\mathcal{M})}^2  
		&
		=
		\int_{\mathcal{M}}
		\bigg|
		\sum_{i=1}^m \lambda_i^n \phi_i
		\bigg|^2
		\leq
		\sum_{i=1}^m 
		\lambda_i^{2n}
		\Vert  \phi_i \Vert_{L^2 (\mathcal{M})}^2
		\leq
		\lambda^{2n} \Vert h\Vert_{L^2 (\mathcal{M})}^2
		,
		\label{EQ64}
	\end{split}
\end{align}
since $\lambda=\max \{\lambda_1,\ldots,\lambda_m \}$ and $\int_M \phi_i \phi_j=0$ for $i,j=1,\ldots, m$ with $i\neq j$.
We now invoke the doubling property \cite[Theorem~4.1]{D}, which provides
constants $r_0,C >0$ such that
\begin{align}
	\Vert h\Vert_{L^\infty (B^g_{2r} (q))}
	\leq
	e^{\gamma}
	\Vert h\Vert_{L^\infty (B^g_{r} (q))}
	\comma
	r\in (0,r_0]
	\comma
	q\in \mathcal{M},
	\label{EQ60}
\end{align}
where $\gamma=C \sqrt{\lambda} + C m^2 \log m +C$.
Note that the paper \cite{D} gives \eqref{EQ60} in the analytic case,
but it is easy to check that the proof extends to the Gevrey setting
with the constant $C$ depending also on~$\sigma$, which is considered fixed.
Without loss of generality, we may assume that $r_0\leq R/2$. 
Since $\mathcal{M}$ is compact, there exists $p_0\in \mathcal{M}$ such that $|h(p_0)|= \sup_{\mathcal{M}} |h|$.
Since $\mathcal{M}$ is connected,
we may join $p_0$ and $p$ by an overlapping chain of balls with radius $2r_0$.
Using \eqref{EQ60}, we obtain
\begin{align}
	\Vert h\Vert_{L^\infty (\mathcal{M})}
	=
	\Vert h\Vert_{B^g_{r_0} (p_0)}
	\leq
	C^{\gamma}
	\Vert h\Vert_{L^\infty (B^g_{r_0} (p))}
		\leq
	C^{\gamma}
	\Vert h\Vert_{L^\infty (B^g_{R/2} (p))}
	\label{EQ70}
	.
\end{align}
From \eqref{EQ64} and \eqref{EQ70}, it follows that
\begin{align}
	\begin{split}
	\Vert\tilde{\mathcal{A}}^n h\Vert_{L^2 (B^g_R (p))}
	&
	\leq
        C
	\Vert \Delta_M^n h\Vert_{L^2 (\mathcal{M})}
	\leq
	C\lambda^n \Vert h\Vert_{L^2 (\mathcal{M})}
	\leq
	C \lambda^n \Vert h\Vert_{L^\infty (\mathcal{M})}
	\leq
	C^{\gamma}
	\lambda^n 
	\Vert h\Vert_{L^\infty (B^g_{R/2} (p))}
	\\
	&
	\leq
		C^{\gamma}
		\left(
		\frac{(\lambda^{1/2\sigma})^{2n}}{(2n)!} 
		\right)^{\sigma}
		(2n)!^\sigma 
		\Vert h\Vert_{L^2 (B_{R/2}^g (p))}
		\leq
				C^{\gamma}
				e^{\sigma \lambda^{1/2\sigma}}
		(2n)!^\sigma 
		\Vert h \Vert_{L^2 (B_{R/2}^g (p))}
		\\&
		\leq
		C^{\gamma}
		(2n)!^\sigma 
		\Vert h \Vert_{L^2 (B_{R/2}^g (p))}
		\comma
		n\in\mathbb{N}_0,
	\label{EQ71}
	\end{split}
\end{align}
since $\lambda\geq 1$ and $\sigma\geq 1$.
Let $\phi(x)=h(x)/ (C^{\gamma} \Vert h\Vert_{L^2 (B_{R/2}^g (p))})$.
Then by \eqref{EQ71}, we obtain
\begin{align}
	\Vert \tilde{\mathcal{A}}^n
	\phi
	\Vert_{L^2 (B_{R}^g (p))}
	\leq
	(2n)!^{\sigma}
	\comma
	n\in \mathbb{N}_0.
	\label{EQ834}
\end{align}
Using \eqref{EQ834} in
Lemma~\ref{L02} yields
\begin{align}
	\Vert \partial^\alpha \phi\Vert_{L^\infty (B_{R/2}^g (p))}
	\leq
	\bar{M}_1
	\bar{M}_2^{|\alpha|}
	|\alpha|!^\sigma
	\comma
	\alpha\in \mathbb{N}_0^d,
	\llabel{EQ143}
\end{align}
where $\bar{M}_1, \bar{M}_2>0$ are constants,
from where
\begin{align}
	\begin{split}
		\Vert \partial^\alpha h\Vert_{L^\infty (B_{R/2}^g (p))}
	&\leq
	C^{\gamma}
	\bar{M}_1
	\bar{M}_2^{|\alpha|}
	|\alpha|!^\sigma
	\Vert h\Vert_{L^\infty (B^g_{R/2} (p))}.
	\label{EQ43}
	\end{split}
\end{align}

We first treat the case $\sigma=1$. 
From \eqref{EQ301}, \eqref{EQ701}, \eqref{EQ60}, and \eqref{EQ43}, we obtain
\begin{align}
\begin{split}
	\Vert h\Vert_{L^\infty (B_{R/2}^g (p))}
	\leq	
	\left(\frac{C}{|E|/N} \right)^{C \gamma}
	\Vert h\Vert_{L^\infty (E)}
	,
	\label{EQ111}
\end{split}
\end{align}
and by \eqref{EQ70} and \eqref{EQ111} we infer that
\begin{align}
		\Vert h\Vert_{L^\infty (\mathcal{M})}
		\leq
		\left(\frac{C}{|E|} \right)^{C \gamma}
		\Vert h\Vert_{L^\infty (E)},
\end{align}
where $C\geq 1$ is a constant depending on~$\mathcal{M}$.

Similarly, if $\sigma>1$, we use \eqref{EQ305}, \eqref{EQ701}, \eqref{EQ60}--\eqref{EQ70}, and \eqref{EQ43} to get
	\begin{align}
	\Vert h\Vert_{L^\infty (\mathcal{M})}
	\leq
	C^{\gamma}
	\Vert h\Vert_{L^\infty (B_{R/2}^g (p))}
	\leq
	\left(\frac{C}{|E|} \right)^{C\gamma}
	\gamma^{C (\sigma-1) \gamma}
	\Vert h\Vert_{L^\infty (E)}
	,
	\llabel{EQ85}
\end{align}
where $C$ is a constant depending on $\sigma$ and~$\mathcal{M}$.
\end{proof}

Finally, we prove the last theorem on observability under the unique continuation condition~\eqref{EQ030}.

\begin{proof}[Proof of Theorem~\ref{T03}]
Let $\epsilon=1/100$. 
For each $z_\alpha \in \partial \Omega$, there exists a local coordinate system $(z', z_d) \in \mathbb{R}^{d-1}\times \mathbb{R}$, a constant $R_\alpha>0$, and a $C^1$ function $\phi_\alpha$ such that:
\begin{enumerate}[label=(\roman*)]
	\item $B_{R_\alpha} (z_\alpha) \cap \Omega =\{ (z',z_d)\in B_{R_\alpha} (z_\alpha): z_d > \phi_\alpha (z')\}$,
	\item $B_{R_\alpha} (z_\alpha) \cap \partial \Omega =\{ (z',z_d)\in B_{R_\alpha} (z_\alpha): z_d = \phi_\alpha (z')\}$, and
	\item $|\phi_\alpha (z_1')- \phi_\alpha (z_2') |\leq \epsilon |z_1'-z_2'|$, for all $(z_1', \phi_\alpha (z_1')), (z_2', \phi_\alpha (z_2')) \in B_{R_\alpha} (z_\alpha)$.
\end{enumerate}
It is clear that $\partial \Omega\subset \cup_{z_\alpha \in \partial \Omega} B_{R_\alpha/2} (z_\alpha)$.
Since $\partial \Omega$ is compact, there exists a finite open cover 
\begin{align}
	\partial \Omega \subset \cup_{i=1}^N B_{R_j/2} (z_j).
	\label{EQ501}
\end{align}
Denote $R_0=\min\{R_1,\ldots,R_N\}/4$ and $\tilde{r}_0=\min \{R_0, r_0,1\}$.
By Lemma~\ref{L00}, there exists a ball $B_r(x)$ with $x\in \Omega$ such that
\begin{align}
	|B_r(x) \cap E|
	\geq 
	\frac{|E|r^d}{C},
\end{align}
for some constant $C\geq 1$. 

Let $r\in (0,\tilde{r}_0]$ and $n\in \mathbb{N}_0$.
First we consider the case when $B_r (x)\cap \Omega\neq \emptyset$.
From \eqref{EQ110} and \eqref{EQ030}, we proceed analogously to~\eqref{EQ163}--\eqref{EQ13}, obtaining
\begin{align}
	\begin{split}
		\Vert f\Vert_{L^\infty (\Omega)}
	&	\leq
		\exp 
		\left(
		\frac{a}{(r/10)^b}
		\right)
		\Vert f\Vert_{L^\infty ( B_r (\tilde{x}))}
		\leq
		2	\exp 
		\left(
		\frac{a}{(r/10)^b}
		\right)
		|f(w)|
		\\	&
		\leq
			\exp \left(
		\frac{a}{(r/10)^b}
		\right)
		C^{n+1}
		\bigl(
		(|\Omega|/|E|)^n \Vert f\Vert_{L^\infty (E)}
		\\&\indeq\indeq
		+
		M (1/\delta)^{n+1}  r^{n+1}
		(n+1)^{(n+1)(\sigma-1)}
		\Vert f\Vert_{L^\infty (\Omega)}
		\bigr)
         .
		\label{EQ113}
	\end{split}
\end{align}
Note that \eqref{EQ113} holds for any $n\in \mathbb{N}_0$ and $r\in (0,\tilde{r}_0]$. 
Let 
\begin{align}
	r= 
	10
	\left(
	\frac{b}{n+1} \right)^{1/b}
	\label{EQ230},
\end{align}
where $n$ is a sufficiently large constant to be determined below;
in particular, we need the right-hand side of \eqref{EQ230}
to be less than or equal to~$\tilde{r}_0$.
Inserting \eqref{EQ230} into the far-right side of \eqref{EQ113}, we obtain
\begin{align}
	\begin{split}
		\Vert f\Vert_{L^\infty (\Omega)}
		&
		\leq
		C_0
		e^{a/b}
		\left(
		(C_0 e^{a/b} |\Omega|/|E|)^n 
		\Vert f\Vert_{L^\infty (E)}
		+
		M 
		\left(\frac{ C_0 e^{a/b}  b^{1/b} }{ \delta(n+1)^{1/b-\sigma+1} }
		\right)^{n+1}
		\Vert f\Vert_{L^\infty (\Omega)}
		\right),
		\label{EQ134}
	\end{split}
\end{align}
where $C_0\geq 1$ is a constant.
Note that $1/b-\sigma+1>0$.
Let $n_0\in \mathbb{N}_0$ be the largest integer less than or equal to
\begin{align}
	\xi:=
	\frac{1}{\log (2C_0 e^{a/b} |\Omega|/|E|)}
	\log
	\left(
	\frac{ M \Vert f\Vert_{L^\infty (\Omega)}}{\Vert f\Vert_{L^\infty (E)}}
	\right)
	+
	\max
	\biggl\{
	\frac{10^b b}{ \tilde{r}_0^b}
	,
	\left(
	\frac{2C_0 e^{a/b}  b^{1/b} }{\delta}
	\right)^{\frac{1}{1/b-\sigma+1}}
	\biggr\}.
	\llabel{EQ36}
\end{align}
The choice of $n=n_0$ ensures that $r\leq \tilde{r}_0$ in \eqref{EQ230}, and the factor on the right side of \eqref{EQ134} satisfies
\begin{align}
	\frac{C_0 e^{a/b}  b^{1/b}}{\delta (n+1)^{1/b-\sigma+1}} 
	\leq 
	\frac{1}{2}
	\llabel{EQ37}.
\end{align}
Inserting $n=n_0$ into \eqref{EQ134}, we arrive at
\begin{align}
	\begin{split}
		\Vert f\Vert_{L^\infty (\Omega)}
		&
		\leq
		C_0
		e^{a/b}
		\bigl(
		 (C_0 e^{a/b} |\Omega|/|E| )^{n_0} \Vert f\Vert_{L^\infty (E)}
		+
		{M \Vert f\Vert_{L^\infty (\Omega)}}{2^{-(n_0+1)}}
		\bigr)
		\\&
		\leq
		C_0
		 e^{a/b}
		\bigl(	(C_0 e^{a/b} 
		 |\Omega| /|E|)^{\xi} \Vert f\Vert_{L^\infty (E)}
		+
		{M \Vert f\Vert_{L^\infty (\Omega)}}{2^{-\xi}}
		\bigr)
		\leq
		C_1
		\Vert f\Vert_{L^\infty (\Omega)}^\gamma
		\Vert f\Vert_{L^\infty (E)}^{1-\gamma}
		,
		\llabel{EQ35}
	\end{split}
\end{align}
where
\begin{align}
	\gamma
	=
	\frac{\log (e^{a/b}  C_0 |\Omega|/|E|)}{\log (2e^{a/b}  C_0 |\Omega|/|E|)}
	\llabel{EQ51}
\end{align}
and $C_1>0$ is a constant depending on $a,b,\tilde{r}_0, M, \sigma, \delta$, and $|\Omega|/|E|$.
Thus, we obtain
\begin{align}
	\begin{split}
	\Vert f\Vert_{L^\infty (\Omega)}
	\leq
	C_1^{1/(1-\gamma)}
	\Vert f\Vert_{L^\infty (E)},	
	\end{split}
\end{align}
completing the proof of the theorem.
The case when $B_r(x)\cap \Omega=\emptyset$ is similar, and thus we omit the details.
\end{proof}

\section*{Acknowledgments}
IK was supported in part by the
NSF grant
DMS-2205493.

\end{document}